\documentclass[11pt,reqno]{amsart}

\usepackage{amssymb, amscd, amsmath, amsthm, float, graphicx,
  enumerate, fullpage
}

\usepackage{mathrsfs}
\usepackage{fouriernc}\DeclareMathAlphabet{\mathcal}{OMS}{cmsy}{m}{n}

\usepackage[utf8]{inputenc}    
\usepackage[T1]{fontenc}

\newcommand{\hide}[1]{}

\usepackage[colorlinks]{hyperref}
\renewcommand{\MRhref}[2]{\href{https://mathscinet.ams.org/mathscinet-getitem?mr=#1}{#2}}
\usepackage[capitalise,nameinlink,noabbrev]{cleveref}

\DeclareMathOperator{\charac}{char}

\DeclareMathOperator{\cok}{cok}
\DeclareMathOperator{\cone}{{\mathsf{cone}}}

\DeclareMathOperator{\diag}{diag}

\DeclareMathOperator{\Ext}{Ext}

\DeclareMathOperator{\Hom}{Hom}

\DeclareMathOperator{\im}{im}

\DeclareMathOperator{\MCM}{MCM}
\DeclareMathOperator{\MF}{MF}

\DeclareMathOperator{\rank}{rank}

\DeclareMathOperator{\syz}{syz{}}
\DeclareMathOperator{\Tor}{Tor}

\DeclareMathOperator{\uHom}{\underline{\Hom}}

\DeclareMathOperator{\uMCM}{\underline{\MCM}}
\DeclareMathOperator{\uMF}{\underline{\MF}}

\renewcommand{\phi}{\varphi}

\renewcommand{\bar}{\overline}
\renewcommand{\tilde}{\widetilde}
\newcommand{\del}{\partial}
\renewcommand{\mod}{\operatorname{mod}}

\renewcommand{\geq}{\geqslant}

\renewcommand{\to}{\longrightarrow}
\newcommand\lto{{\longrightarrow}}

\newcommand\xto{\ -\negmedspace\negmedspace\negthinspace\xrightarrow}
\newcommand\onto{\twoheadrightarrow}

\newcommand{\xx}{\underline{x}}

\newcommand{\dd}[1]{\frac{\partial}{\partial #1}}

\DeclareMathOperator{\Sing}{Sing}
\DeclareMathOperator{\jac}{jac}
\newcommand{\Dsg}{{\mathrm D}_{\mathrm{sg}}}
\newif\ifmusix
\IfFileExists{musixtex.tex}{\musixtrue}{}
\IfFileExists{musixtex.sty}{\musixtrue}{}
\ifmusix
  
  \newcommand{\dbl}{\text{\raisebox{.5ex}{\scalebox{1.3}{\mus 5}}}}%
\else
  \newcommand{\dbl}{\mathord{\text{\raisebox{.15ex}{\scalebox{.78}{$\times$}}}}}%
  \fi
  
\newcommand{\flfl}[1]{{#1}^{\flat\flat}}

\newcommand{\p}{{\mathfrak{p}}}
\newcommand{\q}{{\mathfrak{q}}}

\newcommand{\n}{{\mathfrak{n}}}

\newcommand{\ZZ}{{\mathbb Z}}

\newcommand{\calf}{{\mathcal F}}
\newcommand{\calg}{{\mathcal G}}

\newcommand{\cali}{{\mathcal I}}

\newcommand{\calu}{{\mathcal U}}

\theoremstyle{plain}
\newtheorem{theorem}{Theorem}

\newtheorem{prop}[theorem]{Proposition}
\newtheorem{proposition}[theorem]{Proposition}

\newtheorem{lemma}[theorem]{Lemma}

\newtheorem{corollary}[theorem]{Corollary}

\newtheorem*{conjecture*}{Conjecture}
\newtheorem*{theorem*}{Theorem}
\newtheorem*{prop*}{Proposition}

\theoremstyle{definition}
\newtheorem*{ack*}{Acknowledgements}

\newtheorem{definition}[theorem]{Definition}

\newtheorem{remark}[theorem]{Remark}
\newtheorem*{notation*}{Notation}

\newtheorem{example}[theorem]{Example}

\newtheorem{sit}[theorem]{}

\numberwithin{theorem}{section}
\numberwithin{equation}{section}

\crefname{prop}{Proposition}{Propositions}

   \makeatletter
   \newcommand{\thmlabel}[2]{\phantomsection\def\@currentlabel{#1}\label{#2}}
   \makeatother

\begin{document}

\title[Characteristic-free Kn\"orrer periodicity]{%
A new proof of characteristic-free Kn\"orrer periodicity}

\author[G.J. Leuschke]{Graham J.\ Leuschke}
\address{Dept.\ of Mathematics, Syracuse University,
Syracuse NY 13244, USA}
\urladdr{\href{http://www.leuschke.org/}{http://www.leuschke.org/}}
\email{gjleusch@syr.edu}


\date{\today}



\begin{abstract}
  Let $R = S/(f)$ be a complete local hypersurface ring and let
  $A = S[\![u,v]\!]/(f+uv)$. Pfister and Popescu showed that the stable categories of maximal
  Cohen-Macaulay modules over $R$ and over $A$ are equivalent, with no assumption on the
  characteristic. We give a short proof. If $(\Phi,\Psi)$ is a matrix factorization of $h$ and
  $\del$ is a derivation, then $(\del\Psi,\del\Phi)$ is a null-homotopy for multiplication by
  $\del h$. Since $u$ and $v$ are partial derivatives of $f+uv$, they act as zero on the stable
  category of $A$, so the Koszul complex on $u$ and $v$ splits in the singularity category, and
  this gives density. For the double branched cover $S[\![z]\!]/(f+z^2)$ the same argument
  yields $2z$: it recovers Kn\"orrer's lemma up to free summands on the double branched cover when $2$ is a unit,
  and we show by an example that this lemma, and periodicity for the iterated double branched
  cover $S[\![z,w]\!]/(f+z^2+w^2)$, fail in characteristic two. We also compute the second
  syzygy over $A$ of a maximal Cohen-Macaulay $R$-module by an explicit resolution, and show that
  the change of basis used there cannot be given by a universal matrix formula.
\end{abstract}

\dedicatory{Dedicated to the memory of Ragnar-Olaf Buchweitz}

\maketitle

\section{Introduction}\label{sect:intro}

Let $(S, \n, k)$ be a complete regular local ring, $f\in \n^2$ a non-zero element, and $R = S/(f)$. Kn\"orrer periodicity~\cite[Thm.~3.1]{Knorrer} states that when $k$ is algebraically closed of characteristic different from $2$, the stable categories of maximal Cohen-Macaulay (MCM) modules over $R$ and over $R^{\sharp\sharp} = S[\![z,w]\!]/(f+z^2+w^2)$ are equivalent. The proof passes through the double branched cover $R^\sharp = S[\![z]\!]/(f+z^2)$, and that is where the characteristic enters: one step diagonalizes a matrix by a change of basis of determinant $2$. The change of variables $u = z+iw$, $v=z-iw$ identifies $R^{\sharp\sharp}$ with the hyperbolic extension $A = S[\![u,v]\!]/(f+uv)$ as long as $2$ is a unit, but Kn\"orrer's proof does not apply to give an equivalence between $\uMCM R$ and $\uMCM A$ if $2$ is not a unit.  Indeed, in characteristic $2$ the rings $A$ and $R^{\sharp\sharp}$ differ, and periodicity for $R^{\sharp\sharp}$ fails (see \cref{prop:818-false}).

Pfister and Popescu proved that $\uMCM R \simeq \uMCM A$ with no hypothesis on $S$ or $f$~\cite[Thm.~3.12]{Pfister-Popescu:1996}. This note gives a short proof of Pfister and Popescu's theorem which shows where the $2$ comes from. If $(\Phi,\Psi)$ is a matrix factorization of $h$ and $\del$ is a derivation, differentiating $\Phi\Psi = h\cdot 1$ shows that $(\del\Psi,\del\Phi)$ is a null-homotopy for multiplication by $\del h$, so $\del h$ acts as zero on the stable category.  For $h = f+uv$ we have $\del_v h = u$ and $\del_u h = v$; so $u$ and $v$ act as zero, the Koszul complex on $u, v$ splits in the singularity category of $A$, and every MCM $A$-module is a direct summand, in the singularity category, of a MCM $R$-module regarded as $A$-module by restriction of scalars. For $h = f+z^2$ the derivative is $2z$. When $2$ is a unit the same argument proves Kn\"orrer's lemma on $R^\sharp$ (\cref{prop:up-down-jacobian}); when $2 = 0$ it proves nothing, and the lemma is false (\cref{prop:818-false}).

\cref{sect:notation} fixes the setting and states the theorem, with consequences for Cohen-Macaulay representation type. \cref{sect:proof-pp} proves it: full faithfulness is due to Pfister-Popescu and to Solberg~\cite{Solberg}, and density (\cref{prop:density}) is the argument above. \cref{sect:explicit} computes $\syz_2^A M$ for a MCM $R$-module $M$ by an explicit resolution, in which the derivatives of $(\Phi,\Psi)$ appear as a change of basis, and shows that no universal polynomial formula can replace them. \cref{sect:char2} treats the double branched cover: the Jacobian proof of Kn\"orrer's lemma when $2$ is a unit, and the example $f = x^2$ in characteristic $2$, where the hyperbolic extension shows periodicity and the double branched cover and its iterate do not.

\begin{ack*}
  This paper grew out of an exchange with the LLM Claude about whether the
  characteristic ${} \neq 2$ hypothesis in Kn\"orrer periodicity is essential. The model
  also compiled an initial draft. The author is responsible for all statements, proofs, and
  errors. An earlier version claimed priority for the main result, before the author was alerted
  to the previous work of~\cite{Pfister-Popescu:1996}. 

  I was asked the initial question --- Is there a direct characteristic-free proof of the
  equivalence $\MF(f) \simeq \MF(f+uv)$? --- by Osamu Iyama on November 12, 2017, during a very
  pleasant and productive visit to Nagoya. My friend and collaborator Ragnar-Olaf Buchweitz had
  just passed away the day before. I have returned to the question many times over the years since,
  always with the feeling that Ragnar would have had deeper insight into the question than I
  was able to. I am pleased that this proof uses so many of Ragnar's favorite tools. 
\end{ack*}

\section{The hyperbolic extension}\label{sect:notation}

\begin{sit}
  Throughout, $(S, \n, k)$ is a complete regular local ring of dimension $d+1 \geq 1$, $f$ is a
  non-zero element of $\n^2$, and
  \[
    R = S/(f)\,, \qquad B = S[\![u,v]\!]\,, \qquad A = R^{\dbl} = B/(f+uv)\,. 
  \]
  No hypothesis is placed on $k$; in particular $S$ need not contain a field. We call $A$ the
  \emph{hyperbolic extension} of $R$.

  A \emph{matrix factorization} of $f$ over $S$ is a pair
  $(\phi \colon G \to F, \psi \colon F \to G)$ of homomorphisms of free $S$-modules of the same
  finite rank with $\psi\phi = f\cdot 1_G$ and $\phi\psi = f\cdot 1_F$; then
  $\cok(\phi,\psi) := \cok \phi$ is a MCM $R$-module. The homotopy category $\uMF_S(f)$ of
  matrix factorizations is triangulated, with $\Sigma(\phi,\psi) = (-\psi,-\phi)$ and
  $\cone(\alpha,\beta) = \left(\left[\begin{smallmatrix} \phi' & \alpha \\ 0 &
        -\psi\end{smallmatrix}\right]\,, \left[\begin{smallmatrix}\psi' & \beta \\ 0 &
        -\phi\end{smallmatrix}\right]\right)$, and $\cok$ induces equivalences
  \[
    \uMF_S(f) \simeq \uMCM R \simeq \Dsg(R)
  \]
  with the stable category of MCM $R$-modules and the singularity category
  $\Dsg(R) = {\mathrm D}^{\mathrm b}(\mod R)/\operatorname{perf}(R)$
  \cite[Thm.~4.4.1]{Buchweitz:2021}, \cite[\S 6]{Eisenbud:1980}. The second equivalence holds
  for every Gorenstein ring. Under it, $\Sigma^{-1}$ is $\syz_1$ on stable MCM modules.

  For $\Gamma$ a complete local ring, $\uMCM \Gamma$ is a Krull-Schmidt category: idempotents
  split, and stable MCM $\Gamma$-modules are isomorphic if and only if they are isomorphic in
  $\uMCM \Gamma$ \cite[Theorem 1.8]{Leuschke-Wiegand:BOOK}, \cite[Corollary
  4.4]{Krause:2015}.
\end{sit}

\begin{definition}[{\cite[Definition 8.31]{Leuschke-Wiegand:BOOK}}]
  \label{defn:functor}\ \\ 
  \begin{enumerate}
  \item For $M = \cok(\phi \colon G \to F, \psi \colon F \to G)$ a MCM $R$-module, set
    \begin{equation}\label{eq:knorrer}
      M^{\dbl} =  \cok \left(
        \begin{bmatrix} \phi & -v\,1_F \\ u\,1_G & \psi \end{bmatrix},
        \begin{bmatrix} \psi & v\,1_G \\ -u\,1_F & \phi \end{bmatrix} \right)\,,
    \end{equation}
    a MCM $A$-module. For a morphism $(\alpha, \beta)$ of matrix factorizations, set
    $(\alpha,\beta)^{\dbl} = (\diag(\alpha,\beta),\diag(\beta,\alpha))$.
    
  \item For $N$ a MCM $A$-module, set $N^{\flat\flat} = N/(u,v)N$, a MCM $R$-module since
    $u,v$ is a regular sequence on $N$. Define $(-)^{\flat\flat}$ on morphisms by $u=v=0$.
  \end{enumerate}
\end{definition}

We use the same notation for matrix factorizations. Up to isomorphism, $(-)^{\dbl}$ is
Kn\"orrer's functor $H$ \cite[\S3]{Knorrer} and the functor of Pfister and Popescu
(\cref{rmk:pp-explained}), and Yoshino identifies it with tensoring by the factorization $(u,v)$
of $uv$ \cite[Remark~1.3]{Yoshino:1998}.

\begin{lemma}
  \label{lem:up-down-triangle}
  Let $M$ be a MCM $R$-module. Then
  \begin{enumerate}[\quad(i)]
  \item \label{item:HF} $\flfl{\left(M^{\dbl}\right)} \cong M \oplus \syz_1^R M$; 
  \item \label{item:Fsyz} $\left(\syz_1^R M\right)^{\dbl} \cong \syz_1^A\left(M^{\dbl}\right)$; and
  \item \label{item:triangle} $(-)^{\dbl}$ carries null-homotopies to null-homotopies and
    commutes with suspensions and cones, so is a triangulated functor
    $\uMF_S(f) \to \uMF_B(f+uv)$, equivalently $\uMCM R \to \uMCM A$.
  \end{enumerate}
\end{lemma}

\begin{proof}
  For~\eqref{item:HF}, set $u=v=0$ in \eqref{eq:knorrer}. For~\eqref{item:Fsyz}, conjugation by
  $\diag(1,-1)$ turns the matrix presenting $\left(\syz_1^R M\right)^{\dbl}$ into that presenting
  $\syz_1^A (M^{\dbl})$. For~\eqref{item:triangle}, $(-)^{\dbl}$ commutes with $\Sigma$ on the
  nose, and with $\cone$ up to swapping the two middle blocks of basis vectors.
\end{proof}

\begin{theorem}[{\cite[Theorem 3.12]{Pfister-Popescu:1996}}]
  \label{thm:pp}
  The functor $(-)^{\dbl} \colon \uMCM R \to \uMCM A$ is an equivalence of 
  categories. 
\end{theorem}

\begin{remark}
  \label{rmk:pp-explained}
  Pfister and Popescu work over a Noetherian Henselian local ring $\Lambda$ with an ideal
  $\mathfrak a$ such that $f \notin\mathfrak a$ and $\Lambda/\mathfrak a$ is a discrete
  valuation ring. With $\bar R = \Lambda/(f,\mathfrak a)$, they prove that $(-)^{\dbl}$ is an
  equivalence between the stable categories of modules $M$ over $\Lambda/(f)$ and
  $\Lambda[\![u,v]\!]/(f+uv)$ such that$\Tor_i(\bar R, M) =0$ for $i \geq 1$
  \cite[Definition 1.10, Theorem 3.12]{Pfister-Popescu:1996}. If we take $\Lambda = S$, and let
  $\mathfrak a$ be generated by $x_1, \dots, x_d \in \n$ whose images form a system of parameters
  for $R$ and are linearly independent in $\n/\n^2$, then $S/\mathfrak
  a$ is a DVR, and $f \notin \mathfrak a$ since $R/\mathfrak aR$ has finite length. As
  $\xx$ is $R$-regular and $\xx, u, v$ is $A$-regular, the Tor condition says exactly that $M$
  is MCM. Their functor~\cite[3.11]{Pfister-Popescu:1996},
  $(\phi, \psi) \mapsto \left(\left[\begin{smallmatrix}\phi & v \\ u &
        -\psi\end{smallmatrix}\right], \left[\begin{smallmatrix}\psi & v \\ u &
        -\phi\end{smallmatrix}\right]\right)$, becomes $(\phi,\psi)^{\dbl}$ after changing the
  sign of the second block of basis vectors.
\end{remark}

The following corollary is~\cite[Thm.~8.33]{Leuschke-Wiegand:BOOK} without the standing
hypothesis there that $k$ be algebraically closed of characteristic different from $2$.

\begin{corollary}
  \label{cor:bijection}
  The functor $(-)^{\dbl}$ induces a bijection between the isomorphism classes of
  indecomposable non-free MCM $R$-modules and those of indecomposable non-free MCM
  $A$-modules.
\end{corollary}

\begin{proof}
  Both $\uMCM(R)$ and $\uMCM(A)$ have the Krull-Remak-Schmidt uniqueness property for
  direct-sum decompositions, so the equivalence $(-)^{\dbl}$ induces a bijection on the isomorphism
  classes of the indecomposable objects.
\end{proof}

The next corollary generalizes the main result of \cite{Solberg}, without assuming the existence
of almost split sequences. Combined with Greuel-Kr\"oning \cite{Greuel-Kroning} it gives the
``only if'' direction of \cite[Theorem 9.28]{Leuschke-Wiegand:BOOK}.

\begin{corollary}
  \label{cor:fCMt}
  The hypersurface defined by $f$ has finite (resp., countable, resp., bounded) Cohen-Macaulay
  type if and only if that defined by $f + uv$ does. \qed
\end{corollary}

Nothing here requires $f$ to be reduced or to define an isolated singularity. For example, with
$\charac k = 2$ and $f = x^2y^2 \in k[\![x,y]\!]$, the ring $A$ has a one-dimensional singular
locus, and neither~\cite{Knorrer} nor~\cite{Solberg} applies.

\section{Proof of \texorpdfstring{\cref{thm:pp}}{Theorem 2.4}}\label{sect:proof-pp}

The key computation holds for any matrix factorization and any derivation.

\begin{lemma}
  \label{lem:jacobian}
  Let $T$ be a commutative ring, $h \in T$, and $(\Phi \colon \calg \to \calf, 
  \Psi \colon \calf \to \calg)$ a matrix factorization of $h$ over $T$. Let $\del$ be a
  derivation of $T$, extended entrywise to matrices. Then
  \begin{equation}
    \label{eq:jacobian}
    (\del\Phi)\Psi + \Phi(\del\Psi) = (\del h)\cdot 1_{\calf}
    \qquad\text{and}\qquad
    (\del\Psi)\Phi + \Psi(\del\Phi) = (\del h)\cdot 1_{\calg}\,.
  \end{equation}
  So $(\del\Psi, \del\Phi)$ is a null-homotopy for multiplication by $\del h$, and the Jacobian
  ideal $\jac(h)$ annihilates $\Hom_{\uMF_T(h)}\left((\Phi,\Psi), (\Phi',\Psi')\right)$ for all
  matrix factorizations $(\Phi,\Psi)$, $(\Phi',\Psi')$ of $h$.
\end{lemma}

\begin{proof}
  Differentiate $\Phi\Psi = h \cdot 1_\calf$ and $\Psi\Phi = h \cdot 1_\calg$ by the Leibniz
  rule, and compose.
\end{proof}

\begin{corollary}
  \label{cor:uv-kill}
  Let $N$ and $N'$ be MCM $A$-modules. Then $(u,v)$ annihilates $\uHom_A(N,N')$.
\end{corollary}

\begin{proof}
  Since $f$ involves neither $u$ nor $v$, we have $\dd{u}(f+uv) = v$ and $\dd{v}(f+uv) = u$.
\end{proof}

\begin{remark}
  \label{rmk:jacobian-history}
  \cref{lem:jacobian} is not new. Its conclusion goes back at least to Buchweitz~\cite[Corollary
  7.8.7]{Buchweitz:2021}, who showed that for a complete intersection, the Jacobian ideal annihilates Tate
  cohomology, using the K\"ahler different. H.-J. Wang~\cite[Theorem 5.4]{Wang:1994} and
  Iyengar-Takahashi~\cite[Theorem 4.2]{Iyengar-Takahashi:2021} extend it beyond complete
  intersections using the Noether different, and Herzog-Popescu~\cite[Proposition
  2.5]{Herzog-Popescu:1997} state it for $K[\![\xx]\!]/(f)$, $K$ any field. The proof by
  differentiating appears in~\cite[Lemma 2.4]{Kajiura-Saito-Takahashi:2007}, \cite[\S 4.4 and
  Lemma 4.5]{Dyckerhoff:2011} and~\cite[Lemma 6.3]{Murfet:2013}. I learned it from \"Ozg\"ur
  Esentepe, who used it to show~\cite[Remark 3.1]{Esentepe:2020} that the Jacobian ideal is
  contained in the cohomological annihilator.
\end{remark}

\begin{prop}[{\cite[Lemma 3.14]{Pfister-Popescu:1996}, \cite[Prop.~3.1]{Solberg}}]
  \label{prop:ff}
  The functor $(-)^{\dbl} \colon \uMF_S(f) \to \uMF_B(f+uv)$ is fully faithful.
\end{prop}

\begin{proof}
  This is~\cite[Lemma 3.14]{Pfister-Popescu:1996}, in the setting of \cref{rmk:pp-explained}.
  Earlier, Solberg proved full faithfulness of 
  $F(\phi, \psi) = \left(\left[\begin{smallmatrix} u & \phi \\ \psi &
        -v\end{smallmatrix}\right], \left[\begin{smallmatrix} v & \phi \\ \psi &
        -u\end{smallmatrix}\right]\right)$ for $S = k[\![z_1,\dots,z_r]\!]$; his proof is a
  computation with block matrices that uses only that $B = S[\![u,v]\!]$. The morphism
  $\left(\left[\begin{smallmatrix}0 & 1 \\ 1 & 0\end{smallmatrix}\right], 1\right) \colon
  F(\phi,\psi) \to (\psi,\phi)^{\dbl} \cong \Sigma(\phi,\psi)^{\dbl}$ is an isomorphism, natural
  in $(\phi,\psi)$.
\end{proof}

Let $\cali$ be the essential image of $(-)^{\dbl}$ in $\Dsg(A)$. 

\begin{lemma}
  \label{lem:thick}
  The subcategory $\cali$ is a thick subcategory of $\Dsg(A)$: it is closed under isomorphisms,
  finite direct sums, $\Sigma^{\pm 1}$, cones and direct summands.
\end{lemma}

\begin{proof}
  By \cref{lem:up-down-triangle}\eqref{item:triangle}, $(-)^{\dbl}$ commutes with
  $\Sigma^{\pm1}$ and with cones, and by \cref{prop:ff} every morphism $X^{\dbl}\to Y^{\dbl}$
  is $\theta^{\dbl}$ for some $\theta\colon X\to Y$; so $\cali$ is closed under cones. Likewise
  every idempotent endomorphism of $X^{\dbl}$ is $e^{\dbl}$ for an idempotent $e$ of $X$. Since
  idempotents split in $\Dsg(R)$, $e$ is the projection onto a direct summand $X'$ of $X$, and
  the corresponding summand of $X^{\dbl}$ is $(X')^{\dbl}$.
\end{proof}

\begin{proposition}
  \label{prop:density}
  Every object of $\Dsg(A)$ lies in $\cali$. Equivalently, for every MCM $A$-module $N$
  there is a MCM $R$-module $M$ with $N \cong M^{\dbl}$ in $\uMCM(A)$, and for $N$ stable
  this is an isomorphism of modules.
\end{proposition}

\begin{proof}
  First, every MCM $R$-module $M$, viewed as an $A$-module via $A \onto R$, lies in $\cali$.
  To see this, set $N = M^{\dbl}$. By \cref{lem:up-down-triangle}\eqref{item:HF},
  $\flfl{N} \cong M \oplus \syz_1^R M$. Since $u,v$ is a regular sequence on $N$, the Koszul
  complex
  \[
    0 \lto N \xrightarrow{\left[\begin{smallmatrix} -v \\ u\end{smallmatrix}\right]}
    N \oplus N \xrightarrow{[\,u \;\; v\,]} N \lto 0
  \]
  is quasi-isomorphic to $M \oplus \syz_1^R M$. It is built from $N$ by cones, so lies in
  $\cali$, and hence so does its summand $M$.

  Now let $N$ be a stable MCM $A$-module, representing an arbitrary object of $\Dsg(A)$. By
  \cref{cor:uv-kill}, $u$ and $v$ act as $0$ on every $\uHom_A(\Sigma^i N, \Sigma^j N)$. The
  Koszul complex on $u, v$ over $N$ is acyclic, so in $\Dsg(A)$
  \[
    \flfl N \cong \cone\left(\cone\left(N \xrightarrow{\,u\,} N\right)
    \xrightarrow{\,v\,} \cone\left(N \xrightarrow{\,u\,} N\right)\right)\,.
  \]
  Both maps are zero in $\Dsg(A)$, and the cone of a zero map $X \to Y$ is $Y \oplus \Sigma X$.
  Hence
  \[
    \flfl N \cong N \oplus (\Sigma N)^2 \oplus \Sigma^2 N
  \]
  in $\Dsg(A)$. So $N$ is a direct summand of the MCM $R$-module $\flfl N$, which lies in
  $\cali$ by the first paragraph, and $N$ lies in $\cali$.
\end{proof}

\begin{proof}[Proof of \cref{thm:pp}]
  The triangulated functor $(-)^{\dbl}$ is fully faithful by \cref{prop:ff} and dense by
  \cref{prop:density}.
\end{proof}

\begin{remark}
  \label{rmk:compare-pp38}
  Pfister and Popescu prove density differently~\cite[Cor.~3.7, Lemmas 3.8 and
  3.13]{Pfister-Popescu:1996}. For a unit $\epsilon$ of $S$, $A$ is finite and free over
  $S[\![v-\epsilon u]\!]$, and they show that the Noether different of this extension, which
  contains $v + \epsilon u$, annihilates $\Ext^1_A(N,-)$ for MCM $N$. Two units $\epsilon$ and
  $\epsilon+1$ then give $u$ and $v$; when $k$ is the field with two elements there is no such
  pair, and they first pass to an extension of $S$ with infinite residue field. They then
  split syzygies one variable at a time, over $A/uA = R[\![v]\!]$ and over $A$, and identify
  the result by row and column operations (compare \cref{rmk:syz2}). Here $u$ and $v$ come
  directly from $\del_v$ and $\del_u$, and one Koszul splitting in $\Dsg(A)$ handles both.
\end{remark}

\section{Explicit matrices}\label{sect:explicit}

A MCM $R$-module $M$ is naturally an $A$-module via $A \onto A/(u,v) = R$. This section computes the MCM $A$-module $\syz_2^A M$, which represents $M$ in $\Dsg (A)$, by an explicit resolution. The main result, \cref{thm:syz2}, uses neither \cref{prop:ff} nor \cref{prop:density}.

We need one lemma, which will also be relevant in analyzing \cref{eg:char2}.

\begin{lemma}
  \label{lem:singular-locus}
  Let $T$ be a regular ring, $h \in T$, and $D$ a derivation of $T$. Then $T/(h)$ is regular at
  every prime ideal not containing $Dh$. In particular $\Sing A \subseteq V(u,v)$, and $A$ is a
  normal domain.
\end{lemma}

\begin{proof}
  Let $\q$ be a prime of $T$ with $h \in \q$ and $Dh \notin \q$, and extend $D$ to $T_\q$. If
  $h \in \q^2 T_\q$, write $h = \sum_i a_i b_i$ with $a_i, b_i \in \q T_\q$; then
  $Dh = \sum_i \left((Da_i)b_i + a_i(Db_i)\right) \in \q T_\q$, a contradiction. So
  $h \notin \q^2 T_\q$, and $T_\q/(h)$ is regular.

  For $A$, take $T = B$, $h = f+uv$ and $D = \dd{v}, \dd{u}$, with $Dh = u, v$. Since $u,v$ is
  $A$-regular, $(u,v)$ has height two, so $A$ satisfies $(R_1)$; it is Cohen-Macaulay, so
  satisfies $(S_2)$. By Serre's criterion $A$ is normal, and it is a domain since it is local.
\end{proof}

\begin{theorem}
  \label{thm:syz2}
  Let $M$ be a stable MCM $R$-module, regarded as an $A$-module via $A \onto R$. Then
  \begin{equation}
    \label{eq:syz2}
    \syz_2^A M \cong \left( M \oplus \syz_1^R M\right) ^{\dbl}\,,
  \end{equation}
  and neither side has a non-zero free summand.
\end{theorem}

\begin{proof}
  Let $(\phi\colon G \to F, \psi\colon F \to G)$ be a reduced matrix factorization of $f$ over
  $S$ with $\cok (\phi, \psi) = M$, and let $n$ be the size of $\phi$. As
  in~\cite[Notation 8.13]{Leuschke-Wiegand:BOOK}, set $\tilde F = A \otimes_S F$ and
  $\tilde G = A \otimes_S G$, each isomorphic to $A^n$, and use $\phi$ and $\psi$ also for the
  induced maps between them.

  Since $M \cong F /\phi(G)$ is killed by $u$ and $v$, we have
  $M \cong \tilde F /\left(\phi(\tilde G) + u\tilde F + v \tilde F\right)$, so there is an exact
  sequence
  \[
    \tilde G \oplus \tilde F \oplus \tilde F
    \xto{\alpha}
    \tilde F \to M \to 0\,,
    \qquad
    \alpha = \left[\begin{smallmatrix}\phi & u 1_F & v 1_F\end{smallmatrix}\right]\,,
  \]
  and $\syz_2^AM = \ker \alpha$. Set
  \[
    \beta =
    \left[
      \begin{matrix}
        u & \psi & 0 & -v \\
        -\phi & v & v & 0 \\
        0 & 0 & -u & \phi
      \end{matrix}
    \right]
    \colon
    \tilde G \oplus \tilde F \oplus \tilde F \oplus \tilde G
    \to 
    \tilde G \oplus \tilde F \oplus \tilde F\,.
  \]
  Then $\alpha\beta = 0$. To show $\ker \alpha \subseteq \im \beta$, let $(W, X, Y) \in \ker
  \alpha$ and lift to $\hat W, \hat X, \hat Y \in B^n$, so that
  $\phi(\hat W) + u \hat X + v \hat Y = gH$ for some $H \in B^n$, where $g = f+uv$. Since
  $gH = \phi\psi(H) + uv H$,
  \[
    \phi(\hat W-\psi(H)) + u (\hat X-vH) + v \hat Y = 0
  \]
  in $B^n$. The difference between the two lifts is in the image of the second column of $\beta$,
  so we may assume that 
  \begin{equation}
    \label{eq:Brel}
    \phi( W) + u X + v Y = 0
  \end{equation}
  holds in $B^n$. Setting $u=v=0$ gives $\phi(W) = 0$ in $S^n$; since $\phi$ is injective over
  $S$, $W \in (u,v)B^n$. Write $W = uW_1 + vW_2$. Then
  $u(\phi(W_1)+X) + v(\phi(W_2) + Y) = 0$, and since $u,v$ is a regular sequence on $B^n$ there
  is $Z \in \tilde F$ with $\phi(W_1) + X = vZ$ and $\phi(W_2) +Y = -uZ$. So
  \[
    (W,X,Y) = (uW_1, -\phi(W_1), 0) + (0,vZ,-uZ) + (vW_2, 0, -\phi(W_2))
  \]
  lies in the image of the first, third, and fourth columns of $\beta$.

  Now set
  \[
    \gamma =
    \left[
      \begin{matrix}
        \psi & v & v & 0 \\
        -u & \phi & 0 & v \\
        0 & 0 & \phi& -v \\
        0 & 0 & u & \psi
      \end{matrix}
    \right]
    \colon
    \tilde F \oplus \tilde G \oplus \tilde G \oplus \tilde F
    \to
    \tilde G \oplus \tilde F \oplus \tilde F \oplus \tilde G\,.
  \]
  Then $\beta \gamma =0$, so $\beta$ factors through a surjection $\pi \colon \cok \gamma \to
  \im \beta = \syz_2^A M$. The two diagonal $2n \times 2n$ blocks of $\gamma$ are the matrices
  of \cref{defn:functor} in reverse order: with $(\Phi, \Psi) = (\phi, \psi)^{\dbl}$,
  \[
    \gamma = \left[ \begin{matrix} \Psi & v 1 \\ 0 & \Phi \end{matrix} \right]\,.
  \]
  For $\del = \dd{u}$ we have $\del(f+uv) = v$, and the second equation of
  \cref{lem:jacobian} reads $\left(\del \Psi\right) \Phi + \Psi \left(\del\Phi\right) = v 1$.
  Hence
  \[
    \left[\begin{matrix} 1 & -\del \Psi \\ 0 & 1 \end{matrix} \right]
    \left[\begin{matrix} \Psi & v 1 \\ 0 & \Phi \end{matrix} \right]
    \left[\begin{matrix} 1 & -\del \Phi \\ 0 & 1 \end{matrix} \right]
    =
    \left[\begin{matrix} \Psi & 0 \\ 0 & \Phi \end{matrix} \right]\,,
    \qquad
    \del \Psi = \left[\begin{matrix} 0 & 0 \\ -1 & 0 \end{matrix} \right]\,,
    \quad
    \del \Phi = \left[\begin{matrix} 0 & 0 \\ 1 & 0 \end{matrix} \right]\,,
  \]
  a change of basis defined and invertible over any ring. Therefore
  \[
    \cok \gamma \cong \cok \Psi \oplus \cok \Phi
    = \syz_1^A\left(M^{\dbl}\right) \oplus M^{\dbl}
    \cong \left(\syz_1^R M\right)^{\dbl} \oplus M^{\dbl}\,,
  \]
  a MCM $A$-module without free summands. Since $A$ is a normal domain
  (\cref{lem:singular-locus}) and $\det \gamma = (f+uv)^{2n}$, \cite[Prop.~5.6]{Eisenbud:1980}
  shows that $\cok \gamma$ has rank $2n$.

  Finally, $M$ has rank $0$ over $A$, so the exact sequences
  $0 \to \syz_1^A M \to \tilde F \to M \to 0$ and
  $0 \to \syz_2^A M \to \tilde G \oplus \tilde F \oplus \tilde F \to \syz_1^A M \to 0$ give
  $\rank \syz_1^A M = n$ and $\rank \syz_2^A M = 2n$. Thus $\pi$ is a surjection between
  torsion-free modules of the same rank, hence an isomorphism.
\end{proof}

\begin{remark}
  \label{rmk:syz2}
  Pfister and Popescu show by row and column operations~\cite[Lemma 3.13]{Pfister-Popescu:1996}
  that $\syz_1^A \left(\syz_1^{A/uA} (M) \right) \cong \left( M \oplus \syz_1^R M\right)
  ^{\dbl}$ modulo free summands. Since $A/uA$ has projective dimension $1$ over $A$, it is zero
  in $\Dsg(A)$, so $M \cong \Sigma \syz_1^{A/uA} M$ there and the left side is $\syz_2^A M$ in  $\uMCM A$. Thus their Lemma~3.13 gives \cref{thm:syz2} up to free direct summands. The proof above resolves $M$ over $A$ directly, and shows that neither side has a free summand. 
\end{remark}

The change of basis in the proof of \cref{thm:syz2}, which makes the block-triangular matrix
$\gamma$ block-diagonal, uses the partial derivatives of $(\Phi,\Psi)$. These cannot be replaced
by a polynomial expression in $\Phi$ and $\Psi$:

\begin{prop}
  \label{prop:no-formula}
  Let $f$, $u$, $v$, $\Phi$, $\Psi$ be indeterminates over $\ZZ$ and let
  \[
    \calu = \ZZ[f,u,v]\langle\Phi,\Psi\rangle/\left(\Phi\Psi - (f+uv), \Psi\Phi - (f+uv)\right)
  \]
  be the generic algebra representing the data of a matrix factorization of $f+uv$. Then
  \[
    v \notin \Psi \calu + \calu \Phi\,.
  \]
  Consequently there are no matrices $X$ and $Y$, given by universal polynomial formulas in
  $f$, $u$, $v$, $\Phi$, and $\Psi$ with integer coefficients, such that
  $\left[\begin{smallmatrix}1&X\\0&1\end{smallmatrix}\right]
  \left[\begin{smallmatrix}\Psi&v1\\0&\Phi\end{smallmatrix}\right]
  \left[\begin{smallmatrix}1&Y\\0&1\end{smallmatrix}\right]$ is block diagonal.
\end{prop}

\begin{proof} The algebra $\calu$ is free over $\ZZ[f,u,v]$ on the basis $\left\{1, \Phi, \Psi,
    \Phi^2, \Psi^2, \dots\right\}$, with multiplication determined by
  \[
    \Phi\Phi^a = \Phi^{a+1} \,, \qquad \Psi \Psi^a = \Psi^{a+1}\,, \qquad \Phi \Psi^a =
      (f+uv)\Psi^{a-1} \,, \qquad \Psi \Phi^a = (f+uv) \Phi^{a-1}
  \]
  for $a \geq 1$. So in every element of $\Psi \calu + \calu \Phi$ the coefficient of $1$ lies
  in $(f+uv)\ZZ[f,u,v]$, and $v \notin (f+uv)$.
\end{proof}

The next corollary completes \cref{lem:up-down-triangle}.

\begin{corollary}
  \label{cor:summand}
  Let $N$ be a stable MCM $A$-module. Then
  \[
    \left(\flfl N\right)^{\dbl} \cong N \oplus \syz_1^A N\,;
  \]
  in particular $N$ is a direct summand of $\left(\flfl N\right)^{\dbl}$. 
\end{corollary}

\begin{proof}
  By \cref{prop:density}, $N \cong M^{\dbl}$ for a stable MCM $R$-module $M$, and by
  \cref{lem:up-down-triangle}
  \[
    \left(\flfl N\right)^{\dbl} \cong \left(M \oplus \syz_1^R M\right)^{\dbl}
    \cong M^{\dbl} \oplus \syz_1^A\left(M^{\dbl}\right) = N \oplus \syz_1^A N\,. \qedhere
  \]
\end{proof}

\section{The double branched cover}\label{sect:char2}

Consider the \emph{double branched cover} and its iterate,
\[
  R^\sharp = S[\![z]\!]/(f+z^2)\,, \qquad
  R^{\sharp\sharp} = S[\![z,w]\!]/(f+z^2+w^2)\,.
\]
When $2$ is invertible and $S$ contains $i = \sqrt{-1}$, the change of variables $u = z+iw$,
$v = z-iw$ identifies $R^{\sharp\sharp}$ with $A = R^{\dbl}$. In equal characteristic $2$ the
two rings are not isomorphic in general. There are functors
\[
  (-)^\sharp \colon \MCM(R) \to \MCM(R^\sharp)\,, \quad M^\sharp = \syz_1^{R^\sharp} M\,,
  \qquad\quad
  (-)^\flat \colon \MCM(R^\sharp) \to \MCM(R)\,, \quad N^\flat = N/zN\,.
\]

\begin{prop}[{\cite[Lemma 2.5]{Knorrer}}]
  \label{prop:up-down}
  Let $M$ be a stable MCM $R$-module and $N$ a stable MCM $R^\sharp$-module. Then
  \begin{equation}
    \label{eq:up-down}
    \left(M^{\sharp}\right)^{\flat} \cong M \oplus \syz_1^R M
  \end{equation}
  and, if in addition $\charac k \neq 2$, 
  \begin{equation}\label{eq:818}
    \left(N^\flat\right)^\sharp \cong N \oplus \syz_1^{R^\sharp} N\,.
  \end{equation}
\end{prop}

Kn\"orrer proves \eqref{eq:818} by diagonalizing the matrix $\left[\begin{smallmatrix} -\phi &
    -z \\ z & \phi\end{smallmatrix}\right]$ by a change of basis of determinant $2$.
\cref{lem:jacobian} gives \eqref{eq:818} in $\uMCM R^\sharp$ by the argument of
\cref{prop:density}, and shows how the hypothesis enters.

\begin{prop}
  \label{prop:up-down-jacobian}
  Assume $2$ is a unit in $S$, and let $N$ be a MCM $R^\sharp$-module. Then $z$ annihilates
  $\uHom_{R^\sharp}(N,N')$ for every MCM $R^\sharp$-module $N'$, and
  $\left(N^\flat\right)^\sharp \cong N \oplus \syz_1^{R^\sharp} N$ in $\uMCM R^\sharp$.
\end{prop}

\begin{proof}
  Since $f$ does not involve $z$, $\dd{z}(f+z^2) = 2z$, so by \cref{lem:jacobian} $2z$, and
  hence $z$, annihilates $\uHom_{R^\sharp}(N,N')$. As $R^\sharp/zR^\sharp = R$, the element $z$
  is a parameter on $R^\sharp$, hence regular on $N$, and the exact sequence
  $0 \to N \xrightarrow{\,z\,} N \to N^\flat \to 0$ gives a triangle in $\Dsg(R^\sharp)$ whose
  first map is zero. So $N^\flat \cong N \oplus \Sigma N$ in $\Dsg(R^\sharp)$, and
  \[
    \left(N^\flat\right)^\sharp = \syz_1^{R^\sharp} N^\flat \cong \Sigma^{-1}N \oplus N \cong
    \syz_1^{R^\sharp} N \oplus N
  \]
  in $\uMCM R^\sharp$.
\end{proof}

The same argument, with \cref{cor:uv-kill} in place of $\dd{z}(f+z^2) = 2z$, shows that
$\syz_1^A(N/uN) \cong N \oplus \syz_1^A N$ in $\uMCM A$ for every MCM $A$-module $N$, in every
characteristic; Pfister and Popescu prove this as an isomorphism of modules
\cite[Cor.~3.7(ii)]{Pfister-Popescu:1996}.

When $2 = 0$ in $S$, \cref{lem:jacobian} gives nothing for $R^\sharp$, and the following example
shows that \eqref{eq:818} fails, while the conclusions of \cref{sect:notation,sect:explicit}
hold for $A$.

\begin{example}
  \label{eg:char2}
  Let $k$ be a field of characteristic $2$, $S = k[\![x]\!]$, and $f = x^2$, so that $R =
  k[\![x]\!]/(x^2)$. Then $R$ is a zero-dimensional $(A_1)$ singularity, with exactly two
  indecomposable MCM (equivalently, finitely generated) modules, namely
  \[
    R = \cok (x^2,1) \qquad \text{and} \qquad k = \cok (x,x)\,.
  \]
  The hyperbolic extension $A = k[\![x,u,v]\!]/(x^2+uv)$ is another $(A_1)$ singularity, and is
  an isolated singularity in characteristic $2$ as well, since $\jac(x^2+uv) = (2x,v,u) =
  (u,v)$ is $(u,v,x)$-primary.

  The double branched cover $R^\sharp = k[\![x,z]\!]/(x^2+z^2)$ is not an isolated
  singularity: $x^2+z^2 = (x+z)^2$, so $R^\sharp \cong k[\![s,t]\!] /(s^2)$ is an $(A_\infty)$
  singularity. It therefore has infinite CM type by Auslander's
  theorem~\cite[Thm~7.12]{Leuschke-Wiegand:BOOK}; indeed, the MCM modules over the
  one-dimensional $(A_\infty)$ singularity are classified in~\cite[\S 4]{BGS}: there are the
  two obvious matrix factorizations of size $1$, and
  \[
    \left(
      \left[\begin{matrix}s & t^j \\ 0 & s\end{matrix}\right]\,,
      \left[\begin{matrix}s & -t^j \\ 0 & s \end{matrix}\right]
      \right)
  \]
  for every $j \geq 1$. The classification holds in every characteristic.

  \begin{proposition}
    \label{prop:818-false}
    With the notation of the Example:
    \begin{enumerate}[\quad(i)]
    \item $A$ has exactly two indecomposable MCM modules, namely $A$ and $\p = (u,x)A$. Both
      lie in the image of $(-)^{\dbl}$: explicitly $A = R^{\dbl}$ and $\p = k^{\dbl}$. 
    \item $\left(\flfl{\p}\right)^{\dbl} \cong \p \oplus \syz_1^A \p$ and $\syz_2^A k = \p \oplus \p
        = \left(k \oplus \syz_1^R k\right)^{\dbl}$, in accordance with
        \cref{cor:summand} and \cref{thm:syz2}.
      \item \eqref{eq:818} does not hold. Set $N = R^\sharp/(x+z)R^\sharp$; then $N$ is a stable
        indecomposable MCM $R^\sharp$-module with $\syz_1^{R^\sharp}N \cong N$, but
        $\left(N^\flat\right)^{\sharp}$ is indecomposable. Hence
        $\left(N^\flat\right)^\sharp \not\cong N \oplus \syz_1^{R^\sharp} N$.
      \item The iterated double branched cover $R^{\sharp\sharp} \cong k[\![s,z,w]\!]/(s^2)$
        has infinite CM type, so that in particular $\uMCM R \not \simeq \uMCM R^{\sharp\sharp}$.
    \end{enumerate}
  \end{proposition}

  \begin{proof}
    By \cref{cor:bijection} the indecomposable non-free MCM $A$-modules correspond bijectively
    via $(-)^{\dbl}$ to the indecomposable non-free $R$-modules, of which there is only one. So
    $A$ has exactly two indecomposable MCM modules: $A = R^{\dbl}$ itself, and
    \begin{equation}\label{eq:p}
      k^{\dbl} = \cok(x,x)^{\dbl} = \cok
      \left[\begin{matrix} x & -v \\ u & x\end{matrix}\right]\,,
    \end{equation}
    which is the height-one prime $\p = (u,x)$.

    Setting $u=v=0$ in the presentation matrix of \eqref{eq:p} gives $\diag(x,x)$, so that
    $\flfl{\p} \cong k \oplus k$ and $\left(\flfl{\p}\right)^{\dbl} = \p \oplus \p$. Since
    $\p$ is two-generated of rank one, $\syz_1^A \p$ also has rank one, and is not $A$, so
    $\syz_1^A \p \cong \p$. Over $R$, $\syz_1^R k \cong k$, giving
    $\syz_2^A k = \left(k \oplus \syz_1^R k\right)^{\dbl} = \p \oplus \p$ as in
    \cref{thm:syz2}.

    Consider $N = R^\sharp/(x+z)R^\sharp$ over $R^\sharp \cong k[\![x,z]\!]/((x+z)^2)$. As
    $(x+z,x+z)$ is a reduced matrix factorization of size $1$, $N$ is a stable indecomposable
    MCM $R^\sharp$-module, and $\syz_1^{R^\sharp}N \cong N$ since the factorization is
    symmetric. Setting $z=0$ gives $N^\flat = \cok (x,x) \cong k$ over $R$. Then
    \[
      \left(N^\flat\right)^{\sharp} =
      \cok\left[\begin{matrix} x & -z \\ z & x \end{matrix}\right]
      = \cok\left[\begin{matrix} x & z \\ z & x \end{matrix}\right]
    \]
    as $-1=1$. In the variables $s = x+z$, $t=x$, this becomes
    \[
      \cok\left[\begin{matrix} t & s+t \\s+t & t \end{matrix}\right]
      \cong
      \cok\left[\begin{matrix} t & s+t \\s & t \end{matrix}\right]
      \cong
      \cok\left[\begin{matrix} t & s \\s & 0 \end{matrix}\right]
      \cong
      \cok\left[\begin{matrix} s & t \\ 0 & s \end{matrix}\right]
    \]
    by elementary row and column operations, all invertible in characteristic $2$. By the
    classification in~\cite[\S 4]{BGS}, the last cokernel is indecomposable, while
    $N \oplus \syz_1^{R^{\sharp}}N \cong N \oplus N$ is not.

    Finally, $f+z^2+w^2=(x+z+w)^2$, so $\Sing R^{\sharp\sharp}$ has dimension $2$, and
    $R^{\sharp\sharp}$ has infinite type by Auslander's theorem. Meanwhile $R$ has just one
    indecomposable non-free MCM module, so the stable categories cannot be equivalent.
  \end{proof}
\end{example}

\begin{remark}
  \label{rmk:char2}
  Suppose $2 = 0$ in $S$. Then the derivatives of $f+z^2+w^2$ with respect to $z$ and $w$ are
  $2z = 0$ and $2w = 0$, so \cref{lem:singular-locus} says nothing about $R^{\sharp\sharp}$.
  In fact $f+z^2+w^2 = f+(z+w)^2$, and the change of variables $s = z+w$ gives
  \[
    R^{\sharp\sharp} \cong S[\![s,w]\!]/(f+s^2) = R^\sharp[\![w]\!]\,.
  \]
  The singular locus of $R^{\sharp\sharp}$ thus has dimension one greater than that of $R^\sharp$, and in particular $R^{\sharp\sharp}$ is not an isolated singularity. By Auslander's theorem $R^{\sharp\sharp}$ has infinite CM type for
  every $f$. So $\uMCM R \not\simeq \uMCM R^{\sharp\sharp}$ whenever $R$ has finite CM type,
  and \cref{prop:818-false}(iv) is the case $f = x^2$.
\end{remark}


\def\cprime{$'$} \def\polhk#1{\setbox0=\hbox{#1}{\ooalign{\hidewidth
  \lower1.5ex\hbox{`}\hidewidth\crcr\unhbox0}}}
\providecommand{\bysame}{\leavevmode\hbox to3em{\hrulefill}\thinspace}
\providecommand{\MR}{\relax\ifhmode\unskip\space\fi MR }
\providecommand{\MRhref}[2]{%
  \href{http://www.ams.org/mathscinet-getitem?mr=#1}{#2}
}
\providecommand{\href}[2]{#2}

\end{document}